\documentclass[11pt]{amsart}
\usepackage[english]{babel}
\usepackage{amssymb,amsmath,amsthm,color}
\usepackage[latin2]{inputenc}
\usepackage{amsfonts, verbatim,  amscd}

\usepackage[margin=2.9cm]{geometry}

\def\ol{\overline}

\newtheorem{theorem}{Theorem}

\newtheorem{lemma}[theorem]{Lemma}

\newtheorem{defn}[theorem]{Definition}

\newtheorem{coro}[theorem]{Corollary}
\def\acknowledgment{\par\addvspace{17pt}\small\rmfamily
\trivlist\if!\ackname!\item[]\else
\item[\hskip\labelsep
{\bfseries\ackname}]\fi}

\def\eps{\varepsilon}

\def\C{\mathbb{C}}
\def\R{\mathbb{R}}

\def\N{\mathbb{N}}
\def\D{\mathbb{D}}

\def\cF{\mathcal{F}}

\newcommand{\cO}{\mathcal{O}_J}
\newcommand{\clD}{{\overline{\,\D}}}

\newcommand{\Psh}{\mathrm{Psh}_J}
\def\z{\zeta}%
\begin{document}
\title[On the envelope of Poisson functional]{On the envelope of Poisson functional on almost complex manifolds} 
\author{Florian Bertrand \& Uro\v s Kuzman}

%
%    General info
%
%\subjclass[2000]{32Q60,32Q65,32V40,35J56}
%\date{September 6, 2010}
%\keywords{almost complex manifolds, J-holomorphic disc, maximal totally real manifolds, Riemann-Hilbert problem for Pascali systems}
\begin{abstract}
We establish the plurisubharmonicity of the envelope
of the Poisson functional on almost complex manifolds. That is, we generalize the corresponding result for complex manifolds and almost complex manifolds of complex dimension two. 
\end{abstract} 

\maketitle
We denote by  $\D=\{\zeta\in \C; |\zeta|< 1\}$  the open unit disc in 
the complex plane. Given a smooth almost complex manifold $(M,J)$, 
we denote by $\cO\left(\clD,M\right)$ 
the set of $J$-holomorphic discs in $M$, i.e.,
the set of smooth maps $v\colon \clD \to M$ 
which are $J$-holomorphic in a neighborhood of $\clD$. 

A {\em disc functional} on $M$ is a function 
\[
	H_M \colon \cO\left(\clD,M\right) \to\ol\R= [-\infty,+\infty]. 
\]
The {\em envelope} $EH_M\colon M\to\overline\R$ associated to  $H_M$ in $p\in M$ is 
defined by 
$$EH_M(p)= \inf \left\{ H_M(v); v \in \cO\left(\clD,M\right), \ v(0)=p\right\}.$$ 

Given an upper semi-continuous function $f$ defined on $M$, the corresponding 
{\em Poisson functional}  is the disc functional defined by
\begin{equation*}
	P_f(v)= \frac{1}{2\pi} \int^{2\pi}_0 f\left(v\left(e^{i t}\right)\right)\, {\mathrm d}t,
\end{equation*}
where $v\in \cO\left(\clD,M\right)$. Our main result is the following: 
\begin{theorem}\label{main theorem}
Let $f\colon M\to\R\cup\left\{-\infty\right\}$ be an upper semi-continuous function defined on a smooth almost complex manifold $(M,J)$. Then $EP_f$ is $J$-plurisubharmonic on $M$ or identically $-\infty$.
\end{theorem}
\noindent The envelope $EP_f$ of the Poisson functional 
is the largest plurisubharmonic minorant of the 
upper semi-continuous function $f$. In  case  $M=\C^n$, the corresponding result was proved by 
Poletsky \cite{Poletsky1991,Poletsky1993}, 
and by Bu and Schachermayer \cite{Bu-Schachermayer}.
The result was then extended to complex manifolds
by L\'arusson and Sigurdsson 
\cite{Larusson-Sigurdsson1998,Larusson-Sigurdsson2003}, 
and Rosay \cite{Rosay1,Rosay2}.
Later on, Drinovec Drnov\v{s}ek and Forstneri\v{c} proved it for locally irreducible complex spaces \cite{DF2}. We provide an analogue for almost complex manifolds of any dimension.

In case  $\dim_\mathbb{C}M=2$, Theorem \ref{main theorem} was already established by the second named author \cite{KUZMAN3}. 
The only missing step towards its higher dimensional analogue was a method of attaching a $J$-holomorphic disc to a real torus (often 
referred as a solution to the Riemann-Hilbert problem). More precisely, 
given an embedded  $J$-holomorphic disc $v$, we associate to $v$ a real 
$2$-dimensional torus formed by the boundary circles of discs centered at the boundary points $v(\zeta)$, $\zeta\in\partial\D$. We then
seek a disc centered at $v(0)$ and approximately attached to this torus. In low dimension, such a construction was settled by Coupet, Sukhov and Tumanov based on methods for elliptic PDEs \cite{PROPER, ST2}; however their approach does not apply when 
$\dim_\mathbb{C}M>2$. In the present paper, we then provide, in arbitrary dimension, a solution of the Riemann-Hilbert problem (Theorem 
\ref{propgl}) based on gluing techniques from \cite{BERKUZ,Kuzman2021}. We emphasize that such a result is of an independent interest 
since solutions of such problems appear in various constructions of proper and complete holomorphic maps.

The plan of the paper is as follows. In $\S$1 we introduce the gluing techniques related to \cite{BERKUZ,Kuzman2021}. Section $\S$2 is devoted to the Riemann-Hilbert problem. In $\S$3 we prove Theorem \ref{main theorem}. Finally, we present three applications in $\S$4;  more precisely, we describe some regularization results for $J$-plurisubharmonic functions, we give a characterization of compact $\Psh$-hulls by $J$-holomorphic discs and we establish the $J$-plurisubharmonicity of the envelope of the Lelong functional.

\section{Gluing techniques for $J$-holomorphic maps}
An {\it almost complex structure} $J$ on a real even dimensional smooth manifold $M$ is a $\left(1,1\right)$ tensor field
which  satisfies $J^{2}=-Id$. Throughout the paper we assume that $J$ is smooth. The pair $\left(M,J\right)$ is called an {\it almost complex 
manifold}. We denote by $J_{st}$ the standard integrable structure on $\R^{2n}$ for every $n\in\N$. A differentiable map $v:\left(M',J'\right) \longrightarrow \left(M,J\right)$ between two almost complex manifolds 
is {\it $\left(J',J\right)$-holomorphic} if 
\begin{eqnarray}\label{hol}J\left(v\left(p\right)\right)\circ d_{p}v=d_{p}v\circ J'\left(p\right),\end{eqnarray} 
for every $p \in M'$. When $M'$ is a smoothly bounded domain in $\C$ the map $v$ is called a {\it $J$-holomorphic disc}. We denote by $\D$ the unit disc and by $\Delta$ any other smoothly bounded domain in $\C$.

%An upper semi-continuous function $f$ defined on an open set $V$ in $(M,J)$ is 
%{\it $J$-plurisubharmonic} if $\displaystyle f \circ u$ is subharmonic for any $J$-holomorphic disc $u: \Delta \to V$. We denote by $\Psh(V)$ the set of such functions.  

%Let $u\colon \Delta\to M$ and $v\in \C$. We define the operator
%$$\overline{\partial}_J u(v)=\frac{1}{2}\left(du(v)+J(u)du(i v)\right).$$
%Note that $u$ is a $J$-holomorphic disc if and only if $\overline{\partial}_J u= 0.$ We denote by $u\colon\overline{\Delta}\to M$ the discs that are $J$-holomorphic on some neighborhood of $\overline{\Delta}.$

Let $V\subset M$ and $U\subset \R^{2n}$ be two open sets and let $\phi\colon V\to U$ be a coordinate diffeomorphism. The local structure $J_{loc}=d\phi_*(J)$ can then be represented as a $2n\times 2n$ real matrix map that satisfies $J_{loc}(p)^2=-I$ for all $p\in U$. We denote by $\mathcal{J}(U)$ the set of all such matrix maps satisfying  $\det\left(J_{loc}+J_{st}\right)\neq 0$ on $U$. For such local structures and a map $u=\phi\circ v$, the holomorphicity condition (\ref{hol}) can be rewritten as
\begin{eqnarray}\label{A}\mathcal{F}(u)=u_{\bar{\zeta}}+A(u)\overline{u_{\zeta}}=0,\end{eqnarray}
where $\zeta\in\Delta$ and $A(z)(w)=(J_{st}+J_{loc}(z))^{-1}(J_{loc}(z)-J_{st})(\bar{w})$
is a complex linear endomorphism for every $z\in\C^n$. Hence 
$A$ can be considered as a $n\times n$ smooth complex matrix map acting on $w\in\C^n$ \cite{ST1}. We call $A$ the \emph{complex matrix} associated to $J$ by $\phi$. Throughout the paper, we will use letter $v$ for discs mapping into $M$ and letter $u$ for their local analogues.

Whenever we will work only in $U\subset\mathbb{C}^n$ we will denote the almost complex structure by $J\in\mathcal{J}(U)$ instead of $J_{loc}$. Also we will refer to (\ref{A}) as the $\bar{\partial}_J$-derivative of $u$ and consider $\mathcal{F}$ as a non-linear operator from the Sobolev space $W^{1,p}(\Delta,U)$ to the Lebesgue space $L^{p}(\Delta,U)$ for some $p>2$. We will denote by $\mathcal{O}_J(\Delta,U)$ the set of maps $u\in W^{1,p}(\Delta,U)$  satisfying $\mathcal{F}(u)=0$. Since the boundary of $\Delta$ is assumed to be smooth, the Sobolev embedding theorem implies that $\mathcal{O}_J(\Delta,U)\subset W^{1,p}(\Delta,U)\subset\mathcal{C}^{1-\frac{2}{p}}\left(\overline{\Delta},U\right)$.

The crucial tool in the proof of Theorem \ref{main theorem} is a solution for the  Cousin problem which was developed in \cite{Kuzman2021}. In the following two subsections we present a brief overview for convenience of the reader. 

\subsection{The local Newton-type iteration}
Note that the generalized Cauchy-Riemann system (\ref{A}) is non-linear. Therefore, it is natural to use iterative methods when seeking solutions. Suppose that  $\varphi\colon\Delta\to U$ is a map such that the $L^p$-norm of  $\mathcal{F}(\varphi)$ is small.  Using Newton-type iterations, we seek a disc $u\in\mathcal{O}_{J}(\Delta,U)$ which is $W^{1,p}$-close to $\varphi$. 
For that purpose, we linearize the operator $\mathcal{F}$ along the map $\varphi\in W^{1,p}(\Delta,U)$ and consider  the operator $D_\varphi\colon W^{1,p}(\Delta,\mathbb{R}^{2n})\to L^{p}(\Delta,\mathbb{R}^{2n})$ given by 
$$D_\varphi(V) =  V_{\overline{\zeta}}+A(\varphi)\overline{V_{\zeta}}+\sum_{j=1}^n\left(\frac{\partial A}{\partial z_j}(\varphi)V_j+\frac{\partial A}{\partial \bar{z}_j}(\varphi)\overline{V_j} \right) \ \overline{\varphi_\zeta}.$$     
We point out two important properties of this operator. First, $D_\varphi$ admits a bounded right inverse $Q_\varphi$ for every $\varphi\in W^{1,p}(\Delta,U)$. Moreover, $D_\varphi$ is \emph{locally Lipschitz}, that is, 
%given $\varphi\in W^{1,p}(\Delta,U)$,
 there exists $L>0$ depending on the $W^{1,p}$-norm of $\varphi$, such that for any $\tilde{\varphi} \in W^{1,p}(\Delta,U)$ satisfying $\left\|\varphi-\tilde{\varphi}\right\|_{W^{1,p}(\Delta,U)}<1$ we have 
\begin{equation}\label{eqlip}
\left\|D_\varphi-D_{\tilde{\varphi}}\right\|_{op}<L \left\|\varphi-\tilde{\varphi}\right\|_{W^{1,p}(\Delta,U)},
\end{equation}
where $\|.\|_{op}$ denotes the operator norm. Both properties were established in \cite{BERKUZ} (see Theorem $2$ and p. $7$). 

This enables us to apply the implicit function theorem \cite[Theorem 4]{Kuzman2021}. In particular, we have the following statement. If 
 \begin{eqnarray*}
 \label{bound}\displaystyle \left\|\mathcal{F}(\varphi)\right\|_{L^p(\Delta,U)}<\min\left\{\frac{1}{4\left\|Q_\varphi\right\|_{op}},\frac{1}{8\left\|Q_\varphi\right\|_{op}^2L}\right\},
 \end{eqnarray*} where $L>0$ is given by (\ref{eqlip}), then there exists a disc $u\in W^{1,p}(\Delta,U)$ with $\mathcal{F}(u)=0$ and such that
\begin{eqnarray*}\label{approx}\left\|u-\varphi\right\|_{W^{1,p}(\Delta,U)}\leq2\left\|Q_\varphi\right\|_{op}\left\|\mathcal{F}(\varphi)\right\|_{L^p(\Delta,U)}.\end{eqnarray*}
More precisely, provided that the starting map $u_0=\varphi$ is 'good enough', $u$ arises as the limit map of the iteration 
\begin{eqnarray}\label{Newton}
u_{N+1}=u_N-Q_{\varphi}\mathcal{F}(u_N).
\end{eqnarray}
For details, see \cite[Section 2]{Kuzman2021}. Note also that, by the Sobolev embedding theorem, for $p>2$ the $W^{1,p}$-proximity implies that the limit $u$ is $\mathcal{C}^0$-close to $\varphi$ up to the boundary of $\Delta$.
  
Our strategy is then to use the above iteration method on a family of non-holomorphic maps for which the corresponding bounds are uniform. Therefore, we  introduce the following definition.
\begin{defn}\label{defilqj}
Let $J\in\mathcal{J}(U)$. We say that a family $\mathcal{W}\subset W^{1,p}(\Delta,U)$ \emph{satisfies the LQJ-condition} if
\begin{itemize}
\item[i)] for every $\varphi \in\mathcal{W}$ the  operator $D_{\varphi}$ is locally Lipschitz for a unique constant $L>0$,
\item[ii)] there exists $Q>0$ such that every $D_{\varphi}$, $\varphi \in\mathcal{W}$, admits a right inverse $Q_\varphi$ satisfying $\left\|Q_{\varphi}\right\|_{op}<Q$. 
\end{itemize}
\end{defn}
\noindent It is proved in \cite[Section 2]{Kuzman2021} that every precompact family $\mathcal{W}\subset W^{1,p}(\Delta, U)$ satisfies the LQJ-condition. However, in the present paper we will deal with families that are $W^{1,p}$-unbounded. Therefore, we will have to seek these uniform bounds directly. In particular, we will use the following result which is a direct consequence of Theorem 4 \cite{BERKUZ}.

\begin{coro}\label{local coro}
Let $J\in\mathcal{J}(U)$ and let $\mathcal{W}\subset W^{1,p}(\Delta,U)$ be a family satisfying the LQJ-condition. Then there exist 
$\rho>0$ and $C>0$ depending on $\mathcal{W}$ such that for every $\varphi\in \mathcal{W}$ with $\left\|\mathcal{F}(\varphi)\right\|_{L^{p}(\Delta,U)}<\rho$ there is a $J$-holomorphic map $u\in \mathcal{O}_J(\Delta,U)$ satisfying 
$$\left\|u-\varphi\right\|_{W^{1,p}\left(\Delta,U\right)}<C \left\|\mathcal{F}(\varphi)\right\|_{L^{p}(\Delta,U)}.$$ 
\end{coro}
\noindent This corollary illustrates how these non-linear approximation techniques should be used in practice. In particular, finding a sequence of maps $(\varphi_N)_{N\in \mathbb{N}}$ such that $\|\mathcal{F}(\varphi_N)\|_{L^p}$  is decreasing is not enough. In order to find suitable $J$-holomorphic approximations one has to find 
uniform constants $L$ and $Q$ from Definition \ref{defilqj}; otherwise the size of the neighborhood of $\varphi_N$ on which the implicit function theorem is applied may shrink as $N\to \infty$. That is, $\rho>0$ may tend to zero and $C>0$ may tend to infinity.  

Finally, for any fixed $a\in\Delta$, the right inverses of the corresponding linearized operators $D_{\varphi}$ can be chosen so that $Q_{\varphi}(a)=0$.
% for every $\varphi\in W^{1,p}(\Delta,U)$. 
This implies that the limit map of the iteration also satisfies $u(a)=\varphi(a)$. Indeed, such a pointwise restriction can be obtained by choosing an appropriate variation of the classical Cauchy-Green operator when defining the inverse $Q_{\varphi}$ and is standard in the local approximation theory (see for instance \cite{ST1}).

\subsection{The non-linear Cousin problem}
The above approximation techniques enable us to solve the following local version of the Cousin problem. 

Let $\Delta_1,\Delta_2\subset\C$ be two smoothly bounded and simply connected domains admitting a non-empty simply connected intersection $\Delta_1\cap\Delta_2$  and a smooth cut-off function $\chi\colon\mathbb{C}\to[0,1]$ that is equal to $1$ on a neighborhood of $\overline{\Delta_1\setminus\Delta_2}$ and that vanishes on a neighborhood of $\overline{\Delta_2\setminus\Delta_1}$. In particular, such a function exists if $\overline{\Delta_1\setminus\Delta_2}\cap \overline{\Delta_2\setminus\Delta_1}=\emptyset.$ We call $(\Delta_1,\Delta_2)$ a \emph{good pair}. Let $U\subset \mathbb{C}^n$ be an open set and let $J\in\mathcal{J}(U)$. We consider two $J$-holomorphic discs $u_1\in\mathcal{O}_J(\Delta_1,U)$ and $u_2\in\mathcal{O}_J(\Delta_2,U)$ whose images are $W^{1,p}$-close on $\Delta_1\cap\Delta_2$. The local Cousin problem consists in finding a $J$-holomorphic map that is $\mathcal{C}^0$-close to the connected union of $u_1(\Delta_1)$ and $u_2(\Delta_2)$ (see \cite[Theorem 1]{Kuzman2021} for a precise statement). Its solution can be obtained by defining a pregluing map 
$$\varphi=\chi u_1+(1-\chi)u_2$$ 
on $\Delta_1\cup \Delta_2$. It turns out that the $L^p$-norm of its $\bar{\partial}_J$-derivative is small. Thus this map can be used as the starting map $u_0=\varphi$ in the Newton-type iteration (\ref{Newton}) in order to obtain a map  'gluing together' $u_1$ and $u_2$. 

However, this construction is purely local and cannot be applied in our case since we will need to glue together discs whose images lie in different charts. Such a non-linear version of the Cousin problem was settled in \cite[Theorem 2]{Kuzman2021}. We briefly explain the method. Let $(M,J)$ be an almost complex manifold and let $v_1\colon \Delta_{1}\to M$ and $v_2\colon \Delta_{2}\to M$ be two $J$-holomorphic discs defined on a good pair $(\Delta_1,\Delta_2)$ and with images lying in different chart neighborhoods $\phi^1\colon V_1\to U_1\subset \C^n$ and $\phi^2\colon V_2\to U_2\subset \C^n$, respectively. We assume that the intersection $V_1\cap V_2$ is non-empty and contains the sets $v_1(\Delta_1\cap\Delta_2)$ and $v_2(\Delta_1\cap\Delta_2)$. We also assume that 
$v_1$ and $v_2$ are $W^{1,p}$-close on $\Delta_1\cap\Delta_2$. Then, as above, we can define a pregluing map $\varphi$ that connects the images of $v_1$ and $v_2$. The restrictions 
$\varphi_1=\phi^1\circ\varphi_{|\Delta_1}$ and $\varphi_2=\phi^2\circ\varphi_{|\Delta_2}$ are then taken as the starting maps of two Newton-type iterations that are now performed simultaneously, each in its own local chart equipped with the respective structures $J_1=d\phi^1_*(J)$ and $J_2=d\phi^2_*(J)$. In general, their limits $u_1\in\mathcal{O}_{J_1}(\Delta_1,U_1)$ and $u_2\in\mathcal{O}_{J_2}(\Delta_2,U_2)$ do not provide a global map $\widehat{v}\colon \Delta_1\cup\Delta_2\to M$ since $(\phi^1)^{-1}\circ u_1\neq (\phi^2)^{-1} \circ u_2$ on $\Delta_1\cap\Delta_2$. However, this can be settled with a subtle construction of appropriate right inverses $Q_{\varphi_1}$ and $Q_{\varphi_2}$ (see p. 4267 and p. 4270 in \cite{Kuzman2021}). Such a map $\widehat{v}$ solves the non-linear Cousin problem.    

In this paper, we apply this construction on a collection of $m$ good pairs. We do this as follows. Let $\Delta_0 \subset \C$ and  $\Delta_j\subset \C$, $1\leq j\leq m$, be a collection of non-empty, smoothly bounded and pairwise disjoint domains such that for each $j,k \in \{1,\ldots,m\}$ we have:
\begin{itemize}
\item[i)] $\Delta_0\cap \Delta_j\neq \emptyset$ and is smoothly bounded and simply connected,
\item[ii)] $\overline{\Delta_0\setminus\Delta_j}\cap \overline{\Delta_j\setminus\Delta_0}=\emptyset,$
\item[iii)] $\Delta_j\cap\Delta_k=\emptyset$ for $k\neq j$.
\end{itemize}
Moreover, we assume that  $0\in\Delta_0$ and that 
$$\mathbb{D}=\Delta_0\cup\left(\cup_{j=1}^m\Delta_j\right).$$
 Since such sets $\Delta_j$ form $m$ good pairs with $\Delta_0$, there is a smooth function $\chi\colon \C\to[0,1]$ that is equal to $1$ on some neighborhood of $\overline{\Delta_0\setminus\cup_{j=1}^m\Delta_j}$ and that vanishes on some neighborhood of $\overline{\cup_{j=1}^m\Delta_j\setminus\Delta_0}$. Therefore, given a collection of maps defined on $\Delta_0$ and $\Delta_j$'s, one can connect them with a map defined on $\mathbb{D}$ provided that their images over $\Delta_0\cap\Delta_j$ lie in the same chart neighborhood. 

In this context, let $v_0\colon \Delta_0\to M$ and $v_j\colon \Delta_j\to M$, $1\leq j\leq m$, be $J$-holomorphic discs. We assume that the image $v_0(\Delta_0)$ lies in a fixed coordinate chart $\phi^0\colon V_0\to U_0$ that intersects $m$ distinct coordinate neighborhoods $\phi^j\colon V_j\to U_j$ of the images $v_j(\Delta_j)$. Then, provided that the images of $v_0$ and $v_j$ are $W^{1,p}$-close on $\Delta_0\cap\Delta_j$, we seek a $J$-holomorphic map which is $\mathcal{C}^0$-close to $v_0$ on $\Delta_0$ and $\mathcal{C}^0$-close to $v_j$'s on $\Delta_j$'s. The following theorem solves this question. The statement is written in terms of local coordinates. That is, the maps $u_0$ and $u_j$ correspond to $\phi^0\circ v_0$ and $\phi^j\circ v_j$ respectively, while the diffeomorphisms $\Psi^j\colon U_j^0\to U_0^j$ represent the transition maps $\phi^0\circ(\phi^j)^{-1}$ between the sets $U_j^0=\phi_j(V_j\cap V_0)\subset U_j$ and $U_0^j=\phi_0(V_j\cap V_0)\subset U_0.$
\begin{theorem}\label{Cousin} 
For $1\leq j\leq m$ let $\Delta_0$, $\Delta_j$, $U_0$, $U_j$, $U_0^j$, $U_j^0$ and $\Psi_j$ be defined as above and let $J_0\in\mathcal{J}(U_0)$ and $J_j\in\mathcal{J}(U_j)$ be such that $J_0=d\Psi^j_*(J_j)$ on $U_0^j$. Moreover, let $u_0\in\mathcal{O}_{J_0}(\Delta_0,U_0)$ be an embedding satisfying $u_0\left(\overline{\Delta_0\cap\Delta_j}\right)\subset U_0^j$ and let $\mathcal{W}_j \subset \mathcal{O}_{J_j}(\Delta_j,U_j)$ $1\leq j\leq m$, be families of $J_j$-holomorphic maps satisfying the $LQJ_j$-condition and such that every $u_j\in \mathcal{W}_j$ satsifies  $u_j\left(\overline{\Delta_0\cap\Delta_j}\right)\subset U_j^0$. Then the following conclusion holds: for every $\epsilon>0$ there is $\delta>0$ such that for any $m$-tupple of maps $u_j\in\mathcal{W}_j$, $1\leq j\leq m$, satisfying $\left\|u_0-\Psi^j(u_j)\right\|_{W^{1,p}(\Delta_0\cap\Delta_j,U_0^j)}<\delta$, there exist $\widehat{u}_0\in \mathcal{O}_{J_0}(\Delta_0,U_0)$ and $\widehat{u}_j\in \mathcal{O}_{J_j}(\Delta_j,U_j)$, $1\leq j\leq m$, with the following properties: 
\begin{itemize}
\item[i)]$\Psi_j(\widehat{u}_j)=\widehat{u}_0$ on $\Delta_0\cap\Delta_j$,  
\item[ii)] $\left\|\widehat{u}_0-u_0\right\|_{W^{1,p}(\Delta_0,U_0)}<\epsilon$ and $\left\|\widehat{u}_j-u_j\right\|_{W^{1,p}(\Delta_j,U_j)}<
\epsilon,$
\item[iii)] $\widehat{u}_0(0)=u_0(0).$
\end{itemize}
\end{theorem}
\noindent The corresponding solution of the non-linear Cousin problem provided by this theorem is $\widehat{v}\in \mathcal{O}_J(\D, M)$ given by $(\phi^0)^{-1}\circ \widehat{u}_0$ on $\Delta_0$ and by $(\phi^j)^{-1}\circ \widehat{u}_j$ on $\Delta_j$'s. Due to the properties of the local holomorphic maps which are of class $W^{1,p}$, $p>2$, it is continuous up to the boundary. 

This theorem is a corollary of \cite[Theorem 2]{Kuzman2021}, hence we omit the proof. However, we point out three refinements that we have made. First, we work on $m\geq 2$ good pairs. This can be obtained with $m+1$ simultaneous local Newton-type iterations or by induction. Secondly, in the original reference, the families $\mathcal{W}_j\subset \mathcal{O}_{J_j}(\Delta_j,U_j)$ are assumed to be $W^{1,p}$-precompact. As explained after the Corollary \ref{local coro} such an assumption implies the $LQJ_j$-conditions which are needed for the convergence of the local Newton-type iterations. However, in the present paper, the sets $\mathcal{W}_j\subset \mathcal{O}_{J_j}(\Delta_j,U_j)$ we will consider are $W^{1,p}$-unbounded. Therefore, we will check the $LQJ_j$-conditions directly. Finally, we have added the pointwise condition $\widehat{u}_0(0)=u_0(0)$. Once again, adding such a  constraint is standard by \cite{ST1}.

\section{Approximately attaching a disc to a torus}
Let $u\in\mathcal{O}_{J_{st}}(\overline{\D},\C^n)$ be an usual holomorphic discs and let $G\colon \overline{\D}^2\to \C^n$ be a smooth family such that
\begin{itemize}
\item[i)] $G(z,0)=u(z)$, $z\in\overline{\D}.$
\item[ii)] The map $w\to G(z,w)$ in $J_{st}$-holomorphic for every $z\in \overline{\D}$.
\end{itemize}
We seek a $J_{st}$-holomorphic disc $h$ centered in $u(0)$ and such that its boundary set $h(\partial\D)$ is close to $G(\partial\D\times \partial\D)$. As mentioned in the introduction, this question is usually referred as a Riemann-Hilbert problem. In this, local and usual holomorphic setting, its solution is simple. We take $\varphi^N(\zeta)=G(\zeta,\zeta^N)$ for large $N\in\mathbb{N}$. Such a map admits the required properties but is not holomorphic. However, its $\bar{\partial}_{J_{st}}$-derivative is $L^p$-small on $\D$, therefore one can approximate it with a holomorphic map $h$ that is $\mathcal{C}^0$-close to $\varphi^N$ on $\overline{\D}$. Indeed, when $J\equiv J_{st}$ we have $A\equiv 0$ hence the $LQJ$-condition presented in \S 1.1 is satisfied trivially.

When passing to the non-linear case the above approach meets an obstruction. In particular, the family of maps $\varphi^N$ is not precompact in $W^{1,p}(\D)$ since the $L^p$-norm of the usual complex derivative of $\varphi^N$ explodes when $N\to\infty$. Thus approximating $\varphi_N$ by a $J$-holomorphic map needs more care. In \cite{PROPER, ST2} a solution of the Riemann-Hilbert problem was constructed directly by taking $h(\zeta)=G\left(\zeta e^{\rho(\zeta)},\zeta^N e^{\gamma(\zeta)}\right)$ for large $N\in \N$ and suitable reparametrization functions $\rho$ and $\gamma$. However, for $(\R^{2n},J)$, $n\geq 3$, this approach fails since the system of PDEs for $\rho$ and $\gamma$ becomes overdeterminated. Therefore, our idea is to work in special coordinates, in which the last column of the complex matrix $A$ vanishes. That is, we construct local charts in which $w\to (z',w)$ is $J$-holomorphic for every $z'\in \C^{n-1}$. We prove that in this setting the family of maps $\varphi^N(\zeta)=(\zeta,0,\ldots,0,\zeta^N)$, $N\in\N$, admits the $LQJ$-condition. Thus, for large $N\in\N$ its elements may be approximated by $J$-holomorphic maps.

Finally, we seek solution of the Riemann-Hilbert problem on a manifold. Thus our proof involves gluing techniques presented in \S 1.2. This particular construction was inspired by the work of Rosay \cite{Rosay2} in the integrable case.

\begin{theorem}\label{propgl}
Let $(M,J)$ be a smooth almost complex manifold. We fix $\epsilon>0$, $m\in\mathbb{N}$ and an embedding $v\in \mathcal{O}_J\left(\overline{\D}, M\right)$. Let $G\colon \partial\D\times\overline{\D} \to M$ be a smooth map given by $G(z,\zeta)=v_z(\zeta)$ where for every $z\in \partial\D$ the fiber $v_z\in \mathcal{O}_J\left(\overline{\D}, M\right)$ is centered at $v_z(0)=v(z)$. For $1\leq j\leq m$ let $I_j\subset \partial\D$ be disjoint closed arcs and let $V_j\subset M$ be open sets such that $\left|\partial\D\setminus\cup_{j=1}^m I_j\right|<\epsilon$ and
$G(I_j,\overline{\D})\subset V_j.$ Moreover, for some open sets $U_j\subset \C^n$, we consider local charts $\phi^j\colon V_j\to U_j$ in which we have the following properties:
\begin{itemize}
\item[i)] the  structure $J_j=d\phi^j_*(J) \in \mathcal{J}(U_j)$; 
\item[ii)] the map $\zeta \mapsto v(\zeta)$ coincides with  $u_0^j(\zeta)=(\zeta,0,\ldots,0)$;
\item[iii)] the map $\zeta \mapsto G(z,\zeta)$ coincides with $\zeta \mapsto (z,0,\ldots,\zeta)$ for $z\in I_j$;
\item[iv)] the fibers $\zeta \mapsto(z_1,z_2,\ldots,z_{n-1},\zeta)$ are $J_j$-holomorphic, that is, the complex matrix $A_j$ associated to $J$ by $\phi^j$ has a vanishing last column. 
\end{itemize} 
Then there exist a set $E\subset\partial\D$ of measure $|E|<\epsilon$ and a map $h\in \mathcal{O}_J\left(\overline{\D}, M\right)$ such that $h(0)=v(0)$ and $\textrm{dist}(h,G(\partial\D \times\partial \D))<\epsilon$ on $\partial\D\backslash E$. 
\end{theorem}
\noindent As pointed out, opposed to the above mentioned results \cite{PROPER, ST2}, our theorem requires the existence of special coordinate neighborhoods $\phi^j$ in which the corresponding families satisfy the $LQJ_j$-conditions. These are needed in order to apply Theorem \ref{Cousin} and can not be established without property iv) in the above theorem. Moreover, in contrast with \cite{PROPER, ST2}, the boundary  $h(\partial \D)$  of the disc we obtain is not entirely approximately attached to the given torus. Indeed, the values $h(E)$ may be far from $G(\partial\D\times\partial\D)$. Nevertheless, since $|E|$ is arbitrarily small such a solution of the Riemann-Hilbert problem is enough to establish Theorem \ref{main theorem}.

\begin{proof}
Fix $\epsilon>0$. For $\alpha>0$ we define the sets
$$\Delta_0=\left\{ z\in \mathbb{D}; \textrm{dist}\left(z,\cup_{j=1}^m I_j\right)>\alpha\right\},
\Delta_j=\left\{ z\in \mathbb{D}; \textrm{dist}\left(z,I_j\right)<2\alpha\right\}.$$
We set $\alpha$ small enough to ensure that the sets $\Delta_j$, $1\leq j\leq m$, are pairwise disjoint and that $v(\Delta_j)\subset V_j$. Moreover, we perturb their boundary slightly at the edges, so that we obtain $m$ smoothly bounded good pairs $(\Delta_0,\Delta_j)$ with the properties required in Theorem \ref{Cousin}.

Next, we make a transition into the local coordinates provided by $\phi^j$. In such coordinates, we have the local structure $J_j=d\phi^j_*(J)$, the corresponding complex matrix $A_j$ and  the generalized Cauchy-Riemann operator $\mathcal{F}_j$ defined in $(\ref{A})$. The image of the maps that we are working with are assumed to lie in $\phi^j(V_j)$. Therefore we simplify our notation concerning the spaces of functions and, e.g.  write $W^{1,p}(\Delta_j)=W^{1,p}(\Delta_j,\phi^j(V_j))$. Moreover, since the last column of $A_j$   vanishes, we introduce the notation $A_j=\left[A_j',0\right]\in\mathbb{C}^{n\times n}$ and $u=\left[u',u_n\right]^T\in\mathbb{C}^n$, where we have $A_j'\in \mathbb{C}^{n\times(n-1)}$, $u'\in\mathbb{C}^{n-1}$ and $u_n\in\mathbb{C}$. Note that we have
$$\mathcal{F}_j(u)=u_{\bar{\zeta}}+A_j'(u)\overline{u'_\zeta}=0.$$
Moreover, the linearization of $\mathcal{F}_j$ along $\varphi\in W^{1,p}(\Delta_j)$ is given by
$$D_{\varphi}(V)= V_{\bar{\zeta}}+A'_j(\varphi)\overline{V'_\zeta}+\sum_{k=1}^n\left(\frac{\partial A'_j}{\partial z_k}(\varphi)V_k'+\frac{\partial A'_j}{\partial \bar{z}_k}(\varphi)\overline{V_k}'\right)\overline{\varphi'_{\zeta}}.$$
The following lemma which holds only in these particular coordinates is crucial.

\begin{lemma}\label{c}
Suppose that we are given a family of maps $\varphi_j\in W^{1,p}(\Delta_j)$ mapping into a compact subset of $\C^{n}$ and satsifying 
$$\left\|\varphi_{j}-u_0^j\right\|_{L^p(\Delta_j)}<c_j\;\;\textrm{ and }\;\; \left\|\varphi'_j-(u_0^j)'\right\|_{W^{1,p}(\Delta_j)}<c_j.$$ 
Provided that $c_j>0$ is small enough, this family satisfies the $LQJ_j$-condition.  
\end{lemma}
%\textcolor{red}{
%\begin{lemma}\label{c}
%Suppose that we are given a family of maps $\varphi_j\in W^{1,p}(\Delta_j)$ mapping into a compact subset of $\C^{n}$ and such that 
%$$\left\|\varphi_{j}-u_0^j\right\|_{L^p(\Delta_j)}<c_j\;\;\textrm{ and }\;\; \left\|\varphi'_j-(u_0^j)'\right\|_{W^{1,p}(\Delta_j)}<c_j.$$ 
%Provided that $c_j>0$ is small enough, this family satisfies the $LQJ_j$-condition.  
%\end{lemma}}

\begin{proof}[Proof of Lemma \ref{c}] Let $V\in W^{1,p}(\Delta_j)$. We have
$$D_{\varphi_j}(V)-D_{u_0^j}(V)=I(V)+II(V)+III(V),$$
where
$$I(V)=\left(A_j'(\varphi_j)-A_j'(u_0^j)\right)\overline{V'_\zeta},$$
$$II(V)=\sum_{k=1}^n\left(\frac{\partial A'_j}{\partial z_k}(\varphi_j)\overline{(\varphi_j')_{\zeta}}-\frac{\partial A'_j}{\partial z_k}(u_0^j)\overline{(u_0^j)'_{\zeta}}\right)V_k,$$
$$III(V)=\sum_{k=1}^n\left(\frac{\partial A'_j}{\partial \bar{z}_k}(\varphi_j)\overline{(\varphi'_j)_{\zeta}}-\frac{\partial A'_j}{\partial \bar{z}_k}(u_0^j)\overline{(u_0^j)'_\zeta}\right)\overline{V_k}.$$

Applying the H\"older inequality we get
$$\left\|I(V)\right\|_{L^{p}(\Delta_j)}\leq\left\|A_j'(\varphi_j)-A'_j(u_0^j)\right\|_{L^{2p}(\Delta_j)}\left\|\overline{V_\zeta}\right\|_{L^{2p}(\Delta_j)}.$$
Moreover, $\Delta_j$ is relatively compact hence the $L^{2p}(\Delta_j)$-norm of a given function may be bounded from above by its $L^{p}(\Delta_j)$-norm multiplied by $\textrm{area}(\Delta_j)^{\frac{1}{2p}}$. If we apply this fact twice, we obtain    
$$\left\|I(V)\right\|_{L^{p}(\Delta_j)} \leq\textrm{area}(\Delta_j)^{\frac{1}{p}}\left\|A'_j(\varphi_j)-A'_j(u_0^j)\right\|_{L^{p}(\Delta_j)}\left\|V\right\|_{W^{1,p}(\Delta_j)}.$$
Since $\varphi_j$ is assumed to be $L^p$-close to $u_0^j$, the operator norm of $I(V)$ can be made as small as desired. Indeed, the images of all such $\varphi_j$ are contained in a compact subset of $\mathbb{C}^n$ on which $A_j'$ is a smooth matrix map.   

Furthermore, since $\left\|(u_0^j)'_\zeta\right\|_{L^\infty\left(\overline{\Delta_j}\right)}=1$ we have
$$\left\|II(V)\right\|_{L^{p}(\Delta_j)}\leq \sum_{k=1}^n\left(\left\|\frac{\partial A'_j}{\partial z_k}(\varphi_j)-\frac{\partial A'_j}{\partial z_k}(u_0^j)\right\|_{L^p(\Delta_j)}\left\|V\right\|_{L^{\infty}(\overline{\Delta_j})}\right)+$$
$$+\sum_{k=1}^n\left(\left\|\frac{\partial A'_j}{\partial z_k}(\varphi_j)\right\|_{L^\infty(\overline{\Delta_j)}}\left\|(\varphi'_j)_\zeta-(u_0^j)'_\zeta\right\|_{L^{p}(\Delta_j)}\left\|V\right\|_{L^{\infty}(\overline{\Delta_j})}\right).$$
The $L^p$-norms in the first sum can be assumed to be small by a similar argument as above. The same is true for the second sum since $\varphi_j'$ is $W^{1,p}$-close to $(u_0^j)'$. Hence, by the Sobolev embedding theorem stating that $W^{1,p}(\Delta_j)\subset L^{\infty}\left(\overline{\Delta_j}\right),$ the operator norm of $II(V)$ can be made as small as desired. The same conclusion can be obtained for $III(V)$. 

Provided that $c_j>0$ is small we have proved that $D_{\varphi_j}$ and $D_{u_0^j}$ are arbitrarily close in the operator norm. Let $Q_{u_0^j}$ be the right inverse for $D_{u_0^j}$. Assume that $c_j>0$ is  small enough to ensure 
$$\left\|Id-D_{\varphi_j}Q_{u_0^j}\right\|_{op}\leq \frac{1}{2}.$$
Then the following operator
$$Q_{\varphi_j}=Q_{u_0^j}\left(D_{\varphi_j}Q_{u_0^j}\right)^{-1}$$
is a right inverse for $D_{\varphi_j}$ and satsifies $\left\|Q_{\varphi_j}\right\|_{op}\leq 2\left\|Q_{u_0^j}\right\|_{op}$. We have provided uniformly bounded right inverses $Q_{\varphi_j}$.

Assume now that $\left\|\varphi_j-\tilde{\varphi_j}\right\|_{W^{1,p}(\Delta)}<1$. We seek a bound $L_j>0$ depending only on $c_j>0$ and satisfying
$$\left\|D_{\varphi_j}-D_{\tilde{\varphi_j}}\right\|_{op}\leq L_j\left\|\varphi_j-\tilde{\varphi}_j\right\|_{W^{1,p}(\Delta)}.$$
Such a constant can be provided along the same above computations. The only difference is that the $L^p$-difference between $A_j'(\varphi_j)$ and $A_j'(\tilde{\varphi}_j)$ is not small; it can, however, be  uniformly bounded. The same 
applies to the $L^p$-norm of $\frac{\partial A'_j}{\partial z_k}(\varphi_j)-\frac{\partial A'_j}{\partial \bar{z}_k}(\tilde{\varphi}_j)$. 
\end{proof}

This lemma provides the necessary bounds for local $J_j$-holomorphic approximation and the solution of the non-linear Cousin problem. In particular, our strategy is to define a set of almost holomorphic maps $\varphi_j^{c,N}$, $c\in\partial\D$, $N\in\mathbb{N}$, that are $L^p$-close to $u_0^j$ on $\Delta_j$ and have the desired properties (the additional parameter $c\in\partial \D$ is needed in order to peroform the so-called Poletsky's trick in the proof of the main theorem in \S 3). After that, we will provide their $J_j$-holomorphic approximations and then glue these maps with $v$ using Theorem \ref{Cousin}. 
We point out again that it is  essential that we require  only $\left(\varphi_j^{c,N}\right)'$ to be $W^{1,p}$-close to $\left(u_0^j\right)'$. Indeed, derivatives of the $n$-th component in $\varphi_j^{c,N}$ will grow linearly with $N$.

Let $\Omega_j$, $1\leq j\leq m$, be a range of smooth Jordan domains in $\C$ such that $\Delta_j\subset \Omega_j$, $\partial\Omega_j\cap \partial\Delta_j=I_j$, $v(\Omega_j)\subset V_j$ and $\overline{\Omega_j}\cap \overline{\Delta}_k=\emptyset$ for $k\neq j$. Denote by $h_j$ the Riemann  map between $\overline{\Omega_j}$ and $\overline{\D}$, and  define for $\zeta\in \overline{\Delta}_j$ and $N\in\mathbb{N}$
\begin{equation}\label{hn}
h^N_j(\zeta)=\left(h_j(\zeta)\right)^N.
\end{equation}
These functions are holomorphic in the usual sense and map into $\overline{\mathbb{D}}$. Moreover, note that $|h^N_j|=1$ on $I_j$ and $|h^N_j|\to 0$ uniformly on every compact subset of $\overline{\Delta}_j\setminus I_j$ when $N\to\infty.$    

Fix $c\in\partial \mathbb{D}$. For $\zeta\in \overline{\Delta}_j$, we now define
$$\varphi_j^{c,N}(\zeta)=\left(\zeta,0,\ldots,0, c \cdot h^N_j(\zeta)\right).$$
Due to the properties of $h_j^{N}$, we have the following limit as $N\to\infty$
$$\left\|\varphi_j^{c,N}-u_0^j\right\|_{L^p(\Delta_j)}\to 0.$$
Since $\left(\varphi_j^{c,N}\right)'=(u_0^j)'$, we have 
$$\mathcal{F}_j\left(\varphi_j^{c,N}\right)=A_j\left(\varphi_j^{c,N}\right)\overline{(\varphi_j^{c,N})_\zeta}=A'_j\left(\varphi_j^{c,N}\right)\overline{(u_0^j)'_\zeta}.$$
Recall that the map $u_0^j$ is $J_j$-holomorphic which implies $A'_j(u_0^j)\overline{(u_0^j)'_\zeta}=0$. The images of $\varphi_j^{c,N}$ lie in a compact set on which $A$ is smooth. Hence, the $L^p$-norm of $\mathcal{F}_j\left(\varphi_j^{c,N}\right)$ tends to zero when $N\to\infty$. 

We have now reached the situation in which we can perform the local Newton-type iteration. Indeed, let $c_j>0$ be the constant from Lemma \ref{c}. We can fix $N$ large enough so that for every $c\in\partial \D$ we have
$$\left\|\varphi_j^{c,N}-u_0^j\right\|_{L^{p}(\Delta_j)}<\frac{c_j}{2}.$$ Since such a family of maps satisfies the $LQJ_j$-condition, by Corollary \ref{local coro}, there exist $J_j$-holomorphic maps $u_j^{c,N}\in W^{1,p}(\Delta_j)$ for which the $W^{1,p}$-norm of $u_j^{c,N}-\varphi_j^{c,N}$ is as small as desired. In particular, we need
$$\left\|u_j^{c,N}-\varphi_j^{c,N}\right\|_{W^{1,p}(\Delta_j)}<\epsilon',$$
where $\epsilon'>0$ is associated to $\epsilon>0$  and will be determined below. Moreover, we can assume that 
$$\left\|u_j^{c,N}-u_0^j\right\|_{L^{p}(\Delta_j)}<c_j\;\;\textrm{and}\;\; 
\left\|\left(u_j^{c,N}\right)'-(u_0^j)'\right\|_{W^{1,p}(\Delta_j)}<c_j.$$
This ensures that for large $N$ and every $c\in\mathbb{D}$ the family $\mathcal{W}_j$  of maps $u_j^{c,N}$ also satisfies the $LQJ_j$-condition. Moreover, due to the properties of the maps $h^N_j$, the $W^{1,p}(\Delta_0\cap\Delta_j)$-difference between $u_j^{c,N}$ and $u_0^j$ tends to zero as $N\to\infty$. Indeed, $h_j^N\to 0$ on compact subsets of $\overline{\Delta}_j\setminus I_j$ while the $W^{1,p}$-norm of $u_j^{c,N}-\varphi_j^{c,N}$ tends to zero by Corollary \ref{local coro}. At this stage, everything is set to apply Theorem \ref{Cousin} and solve the non-linear Cousin problem. 

Since $v$ is embedded its image admits a chart neighborhood $\phi^0\colon V_0\to U_0$ such that $d\phi^0_*(J)=J_{st}$ along $\phi^0\circ v\left(\overline{\D}\right)$ and that $\det\left(d\phi^0_*(J)+J_{st}\right)\neq 0$ on $U_0$ (see the appendix of \cite{IR}). Next, we set $u_0=\phi^0\circ v$, $J_0=d\phi^0_*(J)$, $U_0^j=\phi^0(V_0\cap V_j)$, $U_j^0=\phi^j(V_0\cap V_j)$ and $\Psi^j=\phi^0\circ (\phi^j)^{-1}:U_j^0\to U_0^j$. Note that $u_0^j=\Psi_j^{-1}\circ u_0=\phi^j \circ v$ on $\Delta_j$. Therefore, provided that $N$ is large, we can assume that
$$\left\|u_0-\Psi_j(u_j^{c,N})\right\|_{W^{1,p}(\Delta_0\cap\Delta_j)}<\delta,$$
where $\delta>0$ is the constant associated to $\epsilon'>0$ in Theorem \ref{Cousin}. Indeed, the maps $u_j^{c,N}$ are arbitrarily $W^{1,p}$-close to the maps $\varphi_j^{c,N}$, while the last are arbitrarily $W^{1,p}$-close to $u_0^j$ on $\Delta_0\cap\Delta_j$. Note also that the diffeomorphisms $\Psi_j$ are fixed and therefore we have the control over the size of their derivatives. Applying Theorem \ref{Cousin}, we can now provide maps $\widehat{u}_0\in \mathcal{O}_{J_0}(\Delta_0,U_0)$ and $\widehat{u}_j\in \mathcal{O}_{J_j}(\Delta_j,U_j)$, $1\leq j\leq m$, such that  
$$\left\|\widehat{u}_0-u_0\right\|_{W^{1,p}(\Delta_0)}<\epsilon'\textrm{  and  } \left\|\widehat{u}^{c,N}_j-u_j^{c,N}\right\|_{W^{1,p}(\Delta_j)}<
\epsilon'.$$ Moreover, we have $\widehat{u}_0(0)=u_0(0)$ and $\Psi_j(\widehat{u}_j^{c,N})=\widehat{u}_0$ on $\Delta_0\cap\Delta_j$. Thus, the $J$-holomorphic map $h:\overline{\D}\to M$ given by 
\begin{equation*}
h=
\begin{cases}
 \phi_0^{-1}\circ\widehat{u}_0 \ \ \ \  \mbox{on } \Delta_0\\
\phi_j^{-1}\circ \widehat{u}_j^{c,N} \ \ \mbox{on } \Delta_j\\
\end{cases}
\end{equation*}
is well defined and satisfies $h(0)=v(0)$. Finally, we chose $\epsilon'>0$ such that the estimate $$\left\|\widehat{u}^{c,N}_j-\varphi_j^{c,N}\right\|_{W^{1,p}(\Delta_j)}<2\epsilon'$$ implies   
$$\textrm{dist}\left(\phi_j^{-1}\circ \widehat{u}_j^{c,N}\left(\overline{\Delta_j}\right),\phi_j^{-1}\circ \varphi_j^{c,N}\left(\overline{\Delta_j}\right)\right)<\epsilon$$
for all $1\leq j\leq m$. It follows that $h$ is the map we seek. 
\end{proof}

\section{Proof of Theorem 1}
An upper semi-continuous function $f$ defined on an open set $V$ in $(M,J)$ is 
{\it $J$-plurisubharmonic} if 
$\displaystyle f \circ v$ is subharmonic for any $v\in\mathcal{O}_J\left(\overline{\D}, V\right)$. We denote by $\Psh(V)$ the set of $J$-plurisubharmonic
functions on $V$.

The  proof of Theorem \ref{main theorem} follows the one given in \cite{KUZMAN3} but uses the new solution of the Riemann-Hilbert problem presented in Section 2 in order to avoid dimensional restrictions. The reader should note that a large part of the original proof goes through in any dimension. Hence, we drop some details. In particular, we do not prove that $EP_f$ is upper semi-continuous in higher dimensions. This is clear by reading  p. $267$ in \cite{KUZMAN3}. Finally, since Theorem \ref{main theorem} was already proved in real dimension $4$, we will assume that the real dimension of $M$ is greater than $4$.

\begin{proof}[Proof of Theorem \ref{main theorem}]  Let $f$ be an upper semi-continuous function defined on $M$. Since the envelope $EP_f$ is upper semi-continuous it can be approximated from above by continuous functions. Hence we can assume the function $f$ to be
 continuous (see \cite[p. 8]{DF2}). We need to show that for every $p\in M$ with $EP_f(p)>-\infty$ and every smooth 
$J$-holomorphic embedding $v_p \in \mathcal{O}_J\left(\overline{\D}, V\right)$ centered at $p$, we have 
\begin{equation}
\label{to be proved}
 EP_f(p)=EP_f(v_p(0))\leq \int_{0}^{2\pi}
   EP_f\circ v_p(e^{i\theta})\frac{\mathrm{d}\theta}{2\pi}.
\end{equation}

Fix such a disc $v_p$ and a number $\eps>0$. Recall that $v_p$ is defined and $J$-holomorphic on some neighborhood $\mathbb{D}_{v_p}$ of $\overline{\D}$. Let $z'\in \partial\mathbb{D}.$ 
There exists a $J$-holomorphic disc $v_{z'}$ centered at $v_p(z')$ and transverse to $v_p$ which satisfies the following extremality condition
\begin{eqnarray}
\label{ekstremni}
\int_{0}^{2\pi} f\circ v_{z'}\left(e^{it}\right)\frac{\mathrm{d}t}{2\pi} < EP_f(v_p(z')) + \frac{\eps}{4\pi}.
\end{eqnarray}
Since $\dim_{\R}M\geq 6$ we may assume that $v_{z'}$ is embedded \cite{WYS}. Moreover, applying the deformation theory from \cite{ST1} one can perturb $v_{z'}$ and obtain a $J$-holomorphic disc centered at any $q\in M$ that is sufficiently close to $p$. Our strategy is to perform these perturbations in such a way that we will obtain the coordinate charts $\phi^j$ needed to apply Theorem \ref{propgl}. 

Let $V\subset M$ be a small neighborhood of the image $v_{z'}\left(\overline{\mathbb{D}}\right)$. Since $v_{z'}$ and $v_p$ are both embedded one can set a local chart $\phi\colon V\to \mathbb{C}^n$, in which the disc $\phi\circ v_{z'}$ agrees with $u_{z'}(\zeta)=(0,\ldots,0,\zeta)$. Moreover, for values close to $z'$ the disc $v_p$ is identified with the map $z \mapsto (z,0,\ldots,0)$ in such a way that $v_p(z')$ corresponds to the origin. It is proved in \cite[Theorem 1.3.]{ST1} that the manifold of $\phi_*(J)$-holomorphic discs  is a Banach manifold modeled on the space of usual $J_{st}$-holomorphic disc. 
%In particular, in a $W^{1,p}$-neighborhood of $u_{z'}$ there exists a one-to-one correspondence $\Phi$ based on singular integrals and such that $\Phi(u_{z'})=u_{z'}$ and that the map $\Phi(u)$ is $J_{st}$-holomorphic if and only if $u$ is $\phi_*(J)$-holomorphic. Moreover, this correspondence can be constructed to ensure  $\Phi(u)(0)=u(0)$. 
Hence, given a point $q\in\mathbb{C}^{n-1}$ with a small norm, one can define a $J$-holomorphic disc $u_{q}$ centered at $(q,0)$.
%by mapping the translation $\zeta \mapsto(q,\zeta)$ with $\Phi^{-1}$
Moreover, one can obtain a diffeomorphism $H(q,\zeta)\mapsto u_q(\zeta)$ with $\phi_*(J)$-holomorphic $\zeta$-fibers. This allows us to define the desired coordinate chart $\phi^{z'}=H\circ \phi$ on a possibly smaller neighborhood of $v_{z'}\left(\overline{\mathbb{D}}\right)$ contained in $V$. 

If $z\in\partial\mathbb{D}$ is close enough to $z'$ the disc $v^{z'}_z(\zeta)=(\phi^{z'})^{-1}\left(z,0,\ldots,\zeta\right)$ is well defined, centered at $v_p(z)$, and still satisfies the extremality condition $(\ref{ekstremni})$. By compactness, we cover $\partial{\mathbb{D}}$ with a finite number of open arcs 
$$I_{z'}=\left\{ z\in\partial\mathbb{D}; \textrm{ $v_{z}^{z'}$ is well defined}\right\}.$$ 
We shrink these sets into pairwise disjoint closed arcs $I_{z'_1}, I_{z'_2},\ldots, I_{z'_m}$ to ensure that the measure of $E=\partial\mathbb{D}\setminus \cup_{j=1}^mI_{z'_j}$ is small. These arcs are the one required in Theorem \ref{propgl} while  $\phi^{z'_j}$ are the the corresponding charts.

Finally, we define the map $G\colon \partial\D\times\overline{\mathbb{D}}\to M$. Let $J_{z'_1}, J_{z'_2},\ldots, J_{z'_m}\subset \partial\mathbb{D}$ be pairwise disjoint open arcs satsifying $I_{z'_j}\subset J_{z'_j}$. Moreover, we assume that $v_z^{z'_j}$ is defined for every $z\in J_{z'_j}$. Let $\chi\colon\partial\mathbb{D}\to [0,1]$ be a smooth cut-off function that is equal to $1$ on the intervals $I_{z'_j}$ and vanishes off the union of sets $J_{z'_j}$. We set $G(z,\zeta)=v^{z'_j}_z(\chi(z) \zeta)$ for $z\in  J_{z'_j}$ and $G(z,\zeta)=v_p(z)$ elsewhere. That is, for $z\in I_{z'_j}$ we have $G(z,\zeta)=v^{z'_j}_z(\zeta)$, while off these closed arcs we shrink these discs into constant maps. Note that for all $z \in \partial \D$ the disc $\zeta \mapsto G(z,\zeta)$ is $J$-holomorphic. In addition, 
since the measure of $E$ is small and almost every $\zeta$-fiber of $G$ satisfies 
$(\ref{ekstremni})$, we can assume that the following crucial inequality is fulfilled: 
\begin{equation}\label{pravi torus}
	\int_{0}^{2\pi}\!\! \int_{0}^{2\pi} 
	f\circ G(e^{i\theta},e^{it})\frac{\mathrm{d}t}{2\pi}\frac{\mathrm{d}\theta}{2\pi}
	< 	\int_{0}^{2\pi} EP_f \circ v_p(e^{i\theta})\frac{\mathrm{d}\theta}{2\pi}+\frac{\eps}{2}.
\end{equation}
 
Let $h$ be a $J$-holomorphic disc centered at $p$ and with its boundary approximately attached to $G(\partial\D\times\partial\D)$
given by Theorem \ref{propgl}.  By construction the distance between $h(\zeta)$ and $G\left(\zeta,c\cdot h_j^{N}(\zeta)\right)$ is small
for every $\zeta\in I_{z_j'}$, $c=e^{it} \in\partial\D$ and $N\in\N$ large enough. Here $h_j^{N}$ is the map defined in (\ref{hn}). Note that we also have $|h_j^{N}(\zeta)|=1$ for $\zeta \in I_{z_j'}$. Hence for values $e^{i\theta}$ outside the set $E$ we can write $h_j^{N}(e^{i\theta})=e^{i\beta_{N}(\theta)}$. Let us extend $\beta_N$ continuously to a map defined on $[0,2\pi]$ and satisfying $\beta_N(0)=\beta_N(2\pi)$. We claim that    
\begin{equation*}
	\int_{0}^{2\pi} f\circ h(e^{i\theta})\frac{\mathrm{d}\theta}{2\pi}
		< \int_{0}^{2\pi} f\circ G\left(e^{i\theta},e^{i(t+\beta_{N}(\theta))}\right) \frac{\mathrm{d}\theta}{2\pi} +\frac{\eps}{2}.
\end{equation*}
Indeed, for $\theta,t \in [0,2\pi]$ the values of $G\left(e^{i\theta},e^{i(t+\beta_{N}(\theta))}\right)$ lie in the torus $G(\partial\D\times\partial\D)$ while the values of $h(e^{i\theta})$ are by construction contained in a small neighborhood of the set $G\left(\overline{\D}\times\overline{\D}\right)$. Therefore, for $e^{i\theta}\in E$ the difference between these two terms can be uniformly bounded. The rest follows by the fact that the measure of $E$ is as small as desired.  

We now apply the key trick which is due to Poletsky \cite{Poletsky1993}.  We define for $t\in\left[0,2\pi\right]$
\[
	I(t)=\int_{0}^{2\pi} f\circ G\left(e^{i\theta},e^{i(t+\beta_{N}(\theta))}\right)
	\frac{\mathrm{d}\theta}{2\pi}.
\]
According to the mean value theorem there exists $\nu\in\left[0,2\pi\right)$ such that 
\begin{eqnarray*}
%\label{mean-value}
I(\nu) = \int_0^{2\pi}\!\!\!I(t)\;\frac{\mathrm{d}t}{2\pi}=\int_{0}^{2\pi}\!\!\!\!\int_{0}^{2\pi}
	f\circ G\left(e^{i\theta},e^{i(t+\beta_{N}(\theta))}\right)
	\frac{\mathrm{d}t}{2\pi}\frac{\mathrm{d}\theta}{2\pi} 
\end{eqnarray*}
Furthermore, after a change of variables in the double integral, we get
$$I(\nu)= \int_{0}^{2\pi}\!\!\!\!\int_{0}^{2\pi} 
	f\circ G(e^{i\theta},e^{it}) 
	\frac{\mathrm{d}t}{2\pi}\frac{\mathrm{d}\theta}{2\pi}.$$
Together with $(\ref{pravi torus})$ 
we then obtain  
\[
	EP_f(u_p(0))\leq \int_{0}^{2\pi} f\circ h(e^{i\theta})\frac{\mathrm{d}\theta}{2\pi} 
\leq I(\nu)+ \frac{\eps}{2} \leq  \int_{0}^{2\pi} EP_f\circ v_p(e^{i\theta})\frac{\mathrm{d}\theta}{2\pi} + \eps.
\]
Since $\eps>0$ is arbitrary, (\ref{to be proved}) holds and the theorem is proved.  
\end{proof}

\section{Applications}
In this section we provide a few applications of Theorem \ref{main theorem}. The proofs are omitted since they contain very few original ideas. 

\subsection{Regularization of $J$-plurisubharmonic functions} In real dimension four, Pli\'s proved in \cite{plis} that every $J$-plurisubharmonic function can be locally approximated by smooth $J$-plurisubharmonic functions. Along with his almost complex analogue of the Richberg theorem \cite{plis0}, the key ingredient was the low dimensional analogue of our main theorem given in \cite{KUZMAN3}. In fact, as observed by Pli\'s himself, this was the only reason why the result was proved in real dimension four. Therefore, as a corollary of Theorem \ref{main theorem} and \cite{plis0,plis} we obtain the following statement. 
\begin{coro}\label{theoapprox}
Let $(M,J)$ be an almost complex manifold and let $p\in M$. There exists a neighborhood $V$ of $p$ 
such that for every $J$-plurisubharmonic function $f$ on $V$ there exists a decreasing sequence of $J$-plurisubharmonic 
$\mathcal{C}^{\infty}$ functions $(f_N)_{N\in\mathbb{N}}$ on $V$ such that $f_N\rightarrow f$ when $N\to\infty$.
\end{coro}
\noindent It is worth pointing out that Harvey, Lawson and Pli\'s \cite{HLP} proved the corresponding global approximation result in  case when $M$ is $J$-pseudoconvex (see Theorem 4.1 in \cite{HLP}), that is, when $M$ admits a smooth strictly $J$-plurisubharmonic exhaustion function. In particular, they also obtain the above Corollary \ref{theoapprox} (see Corollary 4.2) since any almost complex manifolds admits a system of (strictly) $J$-pseudoconvex neighborhoods. However, their techniques are very different from the ones presented here. 
\subsection{$\Psh$-hull and Poletsky discs}

In addition to providing new $J$-pluri\-subharmonic functions by variational methods, Poletsky theory gives beautiful characterizations of hulls in term of $J$-holomorphic discs. We refer to \cite{Poletsky1993} for the most classical one in $(\C^n,J_{st})$.  
In our case, we can obtain a characterization of compact $\Psh$-hulls of compact sets in terms 
of $J$-holomorphic discs, similar to those obtained in \cite{LP,Larusson-Sigurdsson1998,Rosay1,DF2}. 

Let $K\subset (M,J)$ be a compact set. Its $\Psh$-hull is defined  by
$$\widehat{K}=\{p\in M; f(p)\leq {\rm sup}_K  \ f ,  \  \forall  f \in \Psh(M)\}.$$
Note that, if $M$ is $J$-pseudoconvex, it is enough to consider only continuous or smooth $J$-plurisubharmonic functions. This is due to the above mentioned result of Harvey, Lawson and Pli\'s (see Corollary 6.3 in \cite{HLP}). The following statement is standard and we refer to e.g. \cite{DF2} for its proof.

\begin{coro}\label{corohull}
Let $K\subset (M,J)$ be a compact set whose $\Psh$-hull $\widehat{K}$ is also compact. 
Let $V$ be a relatively compact open set containing $\widehat{K}$. Then $p \in M$ belongs to $\widehat{K}$ if and only if 
for any $\varepsilon>0$ and for any neighborhood $U\subset V$ of $K$, 
there exists a $J$-holomorphic disc $v\in\mathcal{O}_J\left(\overline{\Delta},M\right)$ centered at $p$ 
and a set $E\subset \partial \Delta$ of measure $|E|<\varepsilon$ such that $v(\partial \Delta \setminus E)\subset U$.                            
\end{coro}

\noindent Note that if  $M$ admits a global strictly $J$-plurisubharmonic function, the compactness assumption made on $\widehat{K}$ may be replaced by any of the equivalent conditions stated in Theorem 5.4 of Diederich and Sukhov \cite{DS}; e.g. one obtains a characterization of $\Psh$-hulls of compact sets contained in Stein domains.

In the case when $M$ does not admit any non-constant bounded $J$-plurisub-harmonic function (e.g. if $M$ is compact), Theorem \ref{main theorem} provides 
$J$-holomorphic discs with prescribed center and with most of their boundary in a given open set. This was pointed out by L\'arusson and Sigurdsson in \cite{Larusson-Sigurdsson1998} (see also \cite{Rosay2}). 
\begin{coro}
Let $(M,J)$ be an almost complex manifold admitting no non-constant bounded $J$-plurisubharmonic function. Let $V\subset M$ be a non-empty open set and let $p\in M$. For every $\varepsilon>0$, there exists $v\in\mathcal{O}_J\left(\overline{\Delta}, M\right)$ centered at $p$ and a set $E\subset \partial \Delta$ of measure $|E|<\varepsilon$ such that $v(\partial \Delta \setminus E)\subset V$.             
\end{coro}
\noindent It is worth mentioning that the second named author gave a direct construction of such discs in a certain class of compact almost complex manifolds in which the Runge-type approximation results are possible \cite{KUZMAN2}. 

\subsection{Envelope of the Lelong functional} We consider the {\em Lelong functional} associated to a 
non-negative real function $\alpha$ on $M$, defined by
\begin{equation*}
	L_\alpha(v)= \sum_{\z\in \D} \alpha(v(\z))\,  \log|\z|,
\end{equation*}
where $v\in \cO\left(\clD,M\right)$.
Given a $J$-plurisubharmonic function $f\in \Psh(M)$ and a point $p\in M$
we denote by $\nu_f(p) \in [0,+\infty]$ its {\em Lelong number}
at $p$; in any local coordinate system 
$\phi$ on $M$, with $\phi(p)=0$, we have
\[
	\nu_f(p)= \lim_{r\to 0^+} \frac{\sup_{|\phi(q)|\le r} f(q)}{\log r}.
\]
Note that we consider the constant function $f=-\infty$ as $J$-plurisubharmonic
and set $\nu_{-\infty}=+\infty$.
Given a non-negative function $\alpha\colon M\to \R_+$,  consider
$$ \cF_\alpha=\left\{f \in \Psh(M); f\le 0,\ \nu_f \ge\alpha\right\},$$ 
and the corresponding extremal function on $M$
\begin{equation*}
		f_\alpha=\sup \cF_\alpha.
\end{equation*}
The following theorem establishes the correspondence between $f_\alpha$ and the envelope $EL_\alpha$ of the Lelong functional.
\begin{theorem}\label{main_thm}
Let $\alpha\colon M\to \R_+$ be a non-negative function defined on a smooth almost complex manifold $(M,J)$. Then $EL_{\alpha}$ is $J$-plurisubharmonic and coincides with $f_\alpha$.
\end{theorem} 
\noindent This disc formula was first obtained by Poletsky \cite{Poletsky1993}. For manifolds this result was proved by L\'arusson and Sigurdsson \cite{Larusson-Sigurdsson1998,Larusson-Sigurdsson2003} and by Drinovec Drnov\v{s}ek and Forstneri\v{c} \cite{BDDFF} for locally irreducible complex spaces. We explain its relation with the envelope of the Poisson functional.

If $f\in \cF_\alpha$ and $v\in \cO\left(\clD,M\right)$, then 
$f\circ v\le 0$ is a subharmonic function defined in a neighborhood of $\clD$ whose 
Lelong number at any point $\z\in \D$ satisfies 
\[
	\nu_{f\circ v}(\z) \ge \alpha(v(\z))\, .
\]
Hence $f\circ v$ is bounded above by the largest subharmonic function 
$f_{\alpha \circ v}\le 0$ on $\D$ satisfying $\nu_{f_{\alpha\circ v}}\ge \alpha\circ v$.
This maximal function $f_{\alpha\circ v}$ is the weighted sum of 
Green functions with coefficients $\alpha\circ v$ 
\[
	f_{\alpha\circ v}(z) = \sum_{\z\in \D} \alpha(v(\z)) \, 
		\log\left| \frac{z-\z}{1-\bar \z z}\right|,
		\quad z\in\D.
\]
If the sum is divergent, we take $f_{\alpha\circ v}\equiv-\infty$. 

Setting $z=0$, we note that for every $f\in \cF_\alpha$ and 
$v\in \cO\left(\clD,M\right)$ we have 
\begin{equation*} 
	f(v(0)) 
	\le  \sum_{\z\in \D} \alpha(v(\z)) \log|\z| 
	\le \inf_{\z\in \D} \alpha(v(\z)) \log|\z|=:K_\alpha(v).
\end{equation*}
By taking the infimum over all $J$-holomorphic discs $v$ centered at $p\in M$ we obtain for any $f\in \cF_\alpha$:
\begin{equation}\label{eqenv}
	f(p)\le EL_\alpha(p)
	\le EK_\alpha(p)
	 =: k_\alpha(p).
\end{equation}
It turns out that $f_\alpha$ coincides with  the envelope of the Poisson functional $EP_{k_\alpha}$ for the function $k_\alpha$. Moreover, by taking the supremum over all $f\in\cF_\alpha$ in (\ref{eqenv}) we get $EP_{k_\alpha}=f_\alpha\le E{L}_\alpha$. Thus, in order to prove Theorem \ref{main_thm} it is enough to prove that $E{L}_\alpha\leq EP_{k_{\alpha}}$. 

On almost complex manifolds all the above considerations were settled by the second named author and Drinovec Drnov\v sek \cite{KuzDD}. However, similarly to  \cite{KUZMAN3}, the paper \cite{KuzDD} focuses on real dimension four due to the lack of solutions of the Riemann-Hilbert problem. In particular, proving  $E{L}_\alpha\leq EP_{k_{\alpha}}$ is equivalent to showing that for every continuous function $\varphi\colon M\to \R$ with $\varphi\geq k_\alpha$, embedded $J$-holomorphic disc $v\in\mathcal{O}_J(\clD, M)$, and $\epsilon>0$ there exists a $J$-holomorphic disc $h$ with center  $h(0)=v(0)$ and satisfying
$${L}_{\alpha}(h)=\sum_{z\in\D}\alpha(h(z))\log|z|<\frac{1}{2\pi}\int_0^{2\pi}\varphi\circ v\left(e^{it}\right)\mathrm{d}t+\epsilon.$$
Such a map can now be obtained in any dimension via the new methods developed in Section 2. However, some extra care is needed. In particular, in the proof of Theorem \ref{main theorem} we require that the $J$-holomorphic disc $h(\zeta)$ approximating the map $G(\zeta,c\cdot h_j^N(\zeta))$ satisfies $h(0)=G(0,0)$. In order to prove the above inequality, we must add finitely many such pointwise restrictions (see the proof of Theorem 1 in \cite{KuzDD}). 

\vskip 0,5 cm
\noindent  {\it Acknowledgments.}  Research of the first author was supported by the Center for Advanced Mathematical Sciences at the American University of Beirut. Research of the second named author was supported in part by the research program P1-0291 and the grants J1-3005, J1-1690, N1-0137, N1-0237, US/22-24-079 \ from ARRS, Republic of Slovenia. 
%
%\newpage
%\thanks{}
%\thanks{}

\vskip 0,1 cm
{\small
\noindent Florian Bertrand\\
American University of Beirut, Beirut, Lebanon\\
{\sl E-mail address}: fb31@aub.edu.lb\\
\\
Uro\v{s} Kuzman \\
University of Ljubljana \& University of Maribor, Slovenia\\
{\sl E-mail address}: 	uros.kuzman@fmf.uni-lj.si\\
} 

\end{document}